\documentclass[12pt,leqno,amsfonts,amscd]{amsart}
\usepackage{epsfig}
\usepackage{amsmath}
\usepackage{amsfonts}
\usepackage{amssymb}
\usepackage{color}
\usepackage{hyperref}

\newtheorem{theorem}{Theorem}[section]
\newtheorem{prop}[theorem]{Proposition}
\newtheorem{cor}[theorem]{Corollary}
\newtheorem{lemma}[theorem]{Lemma}
\theoremstyle{definition}

\theoremstyle{remark}

\newtheorem{remark}[theorem]{Remark}

\numberwithin{equation}{section}

\newcommand{\NN}{{\mathbb N}}

\newcommand{\CC}{{\mathbb C}}

\newcommand{\out}[1]{\ }

\let\PSH=\psh

\let\cal=\mathcal
\renewcommand{\phi}{\varphi}

\begin{document}

\title[A note on equality of complex Monge-Amp\`ere measures]{A note on
equality on plurifinely open sets of some complex Monge-Amp\`ere measures}

\author[M. El Kadiri]{Mohamed El Kadiri}
\address{University of Mohammed V
\\Department of Mathematics,
\\Faculty of Science
\\P.B. 1014, Rabat
\\Morocco}
\email{elkadiri30@hotmail.com}




\begin{abstract} Our aim in this paper is to prove that 
if plurisubharmonic functions $u_1,. . . , u_n$, $v_1,. . ., v_n$ 
in the domain of definition of the complex Monge-Amp\`ere 
operator on a domain set $D\subset \CC^n$ 
($n\geq 1$) are such that $u_1= v_1,
. . ., u_n=v_n$ on a Borel plurifinely open set $\Omega\subset D$, then
$$dd^cu_1\wedge ...\wedge dd^cu_n=dd^cv_1\wedge ...\wedge dd^cv_n$$ 
on $\Omega$. This  extends an earlier
result obtained by the author in \cite{EK}.
\end{abstract}
\maketitle

\section{Introduction}
The plurifine topology on an open set $D\subset
\CC^n$, $n\geq 1$, is the coarsest topology on $D$
that makes continuous all plurisubharmonic functions
on $D$. For properties of this topology we refer the reader to
\cite{BT}. An open set relative to this topology is called a
plurifinely open set. 

In  \cite{BT} Bedford and Taylor
proved that if two locally bounded plurisubharmonic functions $u$ and $v$ on a domain
$D\subset \CC^n$ are such  that $u=v$ on a Borel  
plurifinely open set $\Omega \subset D$,
then the restrictions of the Monge-Amp\`ere measures
$(dd^cu)^n$ and $(dd^cv)^n$ to $\Omega$ are equal (see \cite[Corollary 4.3]{BT}). 
In \cite{EK} 
we have extended this result to plurisubharmonic functions $u$
and $v$ in the domain of definition of the complex Monge-Amp\`ere
operator on $D$ (see \cite[Theorem 3.5]{EK}). Our main goal in this paper is to prove the following 
more general result:

\begin{theorem}\label{thm1.1}
Let $D\subset \CC^n$ ($n\geq 1$) be a bounded hyperconvex domain and let
$\Omega$ be a Borel plurifinely open subset of $D$. 
Assume that
$u_1,. . . , u_n, v_1,. . ., v_n$ are plurisubharmonic 
functions in the Cegrell class ${\cal E}(D)$ such that $u_1= v_1,
. . ., u_n=v_n$ on $\Omega$ then
$$dd^cu_1\wedge ...\wedge dd^cu_n|_\Omega=dd^cv_1\wedge ...\wedge dd^cv_n|_\Omega.$$
\end{theorem}

\begin{cor}[{\cite[Theorem 3.5]{EK}}]
Let $D\subset \CC^n$ ($n\geq 1$) be a bounded hyperconvex domain and let
$\Omega$ be a Borel plurifinely open subset of $D$. 
Assume that $u, v\in \cal E(D)$ are such that $u= v$ on $\Omega$ then
$$(dd^cu)^n|_\Omega=(dd^cv)^n|_\Omega.$$
\end{cor}

\begin{remark} 
Let $u$ a plurisubharmonic function on a domain set $D\subset \CC^n$ that is bounded
near the boundary of $D$ and $T$ a closed positive current on $D$
of bidegree $(p,p)$ ($p<n$).
By Proposition 4.1 from \cite{HH} the current $dd^c(u T)$ has locally finite mass  in $D$. 
This allows one to define the current
$dd^cu\wedge T$ on $D$ by putting
$$dd^cu\wedge T=dd^c(uT).$$ 
In \cite{EK} 
we have proved that if $T$ is of bidegree $(n-1,n-1)$ and 
$u=v$ on a Borel plurifinely open set $\Omega\subset D$, then
$$dd^cu\wedge T|_\Omega=dd^cv\wedge T|_\Omega$$ 
see \cite[Theorem 4.5]{EK}. This result is to be compared with Theorem \ref{thm1.1}.
\end{remark}

\section{The cegrell classes and the domain of definition of the
Monge-Amp\`ere operator}\label{section2}

Let $D$ be a bounded hyperconvex domain in $\CC^n$. From \cite{Ce} 
we recall the following subclasses of $\PSH_-(D)$, the cone of
nonpositive plurisubharmonic functions
on $D$:

$${\cal E}_0={\cal E}_0(D)=\{\varphi\in \PSH_-(D): \lim_{z\to \partial \Omega}\varphi(z)=0, \ \int_D (dd^c\varphi)^n<\infty\},$$

$${\cal F}={\cal F}(D)=\{\varphi\in \PSH_-(D): \exists \ \cal E_0\ni \varphi_j\searrow \varphi, \
\sup_j\int_D (dd^c\varphi_j)^n<\infty\},$$

and
\begin{equation*}
\begin{split}
\cal E=\cal E(D)=\{\varphi \in \PSH_-(D): \forall z_0\in \Omega, \exists
\text{ a neighborhood } \omega \ni z_0,&\\
{\cal E}_0\ni \varphi_j \searrow \varphi \text{ on } \omega, \ \sup_j \int_D(dd^c\varphi_j)^n<\infty \}.
\end{split}
\end{equation*}

As in \cite{Ce}, we note that if $u\in \PSH_-(D)$ then $u\in \cal E(D)$ if and only
if for every $\omega\Subset D$, there is $v\in \cal F(D)$ such that $v\ge u$ and
$v=u$ on $\omega$. On the other hand we have $\PSH_-(D)\cap L^\infty_{loc}(D)\subset \cal E(D)$.
The classical Monge-Amp\`ere operator on $\PSH_-(D)\cap L^\infty_{loc}(D)$ can
be extended uniquely to
the class $\cal E(D)$, the extended operator is still denoted by $(dd^c\cdot)^n$.
More generally,
Cegrell has proved that for given functions $u_1,...,u_n$ in the class $\cal E(D)$, 
the Monge-Amp\`ere measure $dd^cu_1\wedge...\wedge dd^cu_n$ is well-defined
(see \cite[Theorem 4.2 and Definition 4.3]{Ce}) and that this measure coincides
with $dd^cu_1\wedge...\wedge dd^cu_n$ as defined in \cite[p. 113]{Kl}
when  $u_1,...,u_n$ are locally bounded. If $u_1=\cdots =u_n=u\in \cal E(D)$,
then $dd^cu_1\wedge...\wedge dd^cu_n$ is the measure $(dd^cu)^n$.
According to Theorem 4.5 from \cite{Ce}, the class $\cal E$ is the biggest class
$\cal K\subset \PSH_-(D)$ satisfying the following conditions:

(1) If $u\in \cal K$, $v\in \PSH_-(D)$ then $\max(u,v)\in \cal K$.

(2) If $u\in \cal K$, $\varphi_j\in \PSH_-(D)\cap L^\infty_{loc}(D)$, $\varphi_j\searrow u$, 
$j\to +\infty$,
then $((dd^c\varphi_j)^n)$ is weak*-convergent.

We also recall, following Blocki, cf \cite{Bl}, that the general domain of definition $\cal D$ of
the Monge-Amp\`ere operator on an arbitrary domain $D$ of $\CC^n$ consists of
plurisubharmonic functions $u$
on $D$ for which there is a nonnegative (Radon) measure $\mu$ on $D$
such that for any open set $U\subset D$ and any decreasing sequence
$(u_j)$ of smooth (that is, of class $\cal C^\infty$)
pluri-subharmonic functions on $U$ converging to $u$, the sequence 
of measures $(dd^cu_j)^n$ is weakly-convergent
to (the restriction to $U$ of) $\mu$. The measure $\mu$ is 
denoted by $(dd^cu)^n$ and called the Monge-Amp\`ere
measure of (or associated with) $u$.
When $D$ is bounded and hyperconvex then $\cal D\cap\PSH_-(D)$ coincides with the
class $\cal F=\cal F(D)$, cf. \cite{Bl}.

Theorem \ref{thm1.1} and its corollary remain true  if $\cal E(D)$ is replaced by
the Blocki domain of definition $\cal D$ of the Monge-Amp\`ere operator
for any domain $D\subset \CC^n$ ($D$ bounded if $n=1$).

\section{Proof of Theorem 1.1}

Let $n\geq 1$ be an integer. For a domain $D$ in $\CC^n$ and a set $E\subset D$, let 
$u_E=u(E,D,\cdot)$ be the function defined on $D$ by 
$$u_E(z):=\sup \{v(z): v\in \PSH_{-}(D), \ v\le -1 \text{ on } E\}$$
and $u_E^*$ the upper semicontinuous regularization of $u_E$, that is the function
defined on $D$ by
$$u_E^*(z)=\limsup_{\zeta\to z}u_E(\zeta)$$
for every $z\in D$, (the relative extremal function of $E$, see \cite[p. 158]{Kl}). 
We clearly have 
$$-1\leq u_E\leq u_E^*\leq 0$$ 
on $D$ and $u_E=u_E^*$ q.e.  on $D$, that is, $u_E=u_E^*$ outisde a pluripolar 
subset of $D$.

A set $E\subset D$ is said to be pluri-thin at a $z\in D$ if 
$z\notin \overline E$ or there is an open set $\sigma\subset D$ 
containing $z$ and a plurisubharmonic function $u$ on $\sigma$  such that 
$$\limsup_{\zeta\to z, \zeta\in E}u(\zeta)<u(z).$$ 
According to Theorem 2.3 from \cite{BT}, a set $\Omega\subset D$ is plurifinely  open
if and only if $\complement \Omega=:D\setminus \Omega$ is pluri-thin at every point $z\in \Omega$.

\begin{prop}[{\cite[Proposition 3.4]{EK}}]\label{prop1.1}
Let $D$ be a domain of $\CC^n$. If a subset $E$ of $D$ is pluri-thin at 
$z\in \CC^n$, then there is an open
set $V_z$ containing $z$ such that
$u_{(E\setminus \{z\})\cap V_z}^*(z)>-1$.
\end{prop}

\begin{proof}
The result is obvious if $z\notin \overline E$. Suppose that $z\in \overline E$ and that
$E$ is pluri-thin at $z$. According to \cite[Proposition 2.2]{BT},
there is a plurisubharmonic function $u\in \PSH_-(D)$ such that
$$\limsup_{\zeta\to z, \zeta\in E, \zeta\ne z}u(\zeta)< u(z).$$
Hence, there is an open set $V_z\subset D$ containing $z$, and a real $\alpha< 0$ such that
$$u(\zeta)\le \alpha <u(z)$$
for every $\zeta\in (E\setminus \{z\})\cap V_z$. The function $v=\frac{-1}{\alpha}u$
is  nonnegative and plurisbharmonic on $\Omega$ and satisfies $v(\zeta)\leq -1$
for every $\zeta \in (E\setminus \{z\})\cap V_z$, so that
$$-1<v(z)\le u_{(E\setminus \{z\})\cap V_z}^*(z).$$
\end{proof}

\begin{lemma}\label{lemma1.1}
Let $D\subset \CC^n$ ($n\geq 2$) be a domain and let $\Omega$
be a plurifinely open subset of $D$. There exists a set
$J\subset \NN$ with the property that for each $j\in J$
there exists an open ball $B_j\subset \overline B_j\subset D$ and a
function  $\varphi_j\in \PSH_+(B_j)$ such that $\Omega\subset \bigcup_{j\in J}\{\varphi_j>0\}$
and that $\bigcup_{j\in J}\{\varphi_j>0\}\setminus \Omega$ is pluripolar.
\end{lemma}

\begin{proof}
Let $\{B'_j: j\in \NN\}$ be a base of the Euclidean topology on $\CC^n$ formed by open
balls relative to the Euclidean norm on $\CC^n$ such that 
$\overline {B'_j}\subset D$ for every $j$, 
and let $z\in \Omega$. The set $\complement \Omega=D\setminus \Omega$ being thin at 
$z$, then, according to Proposition
\ref{prop1.1}, there is an integer $j_z$ such that 
$z\in B'_{j_z}$ and $u^*_{(\complement \Omega) \cap B'_{j_z}}(z)>-1$.
Denoting by $\bar u^*_{(\complement \Omega)\cap B'_{j_z}}$ the function
defined in the same manner as $u^*_{({\complement \Omega})\cap B'_{j_z}}$ with $D$
replaced by $B_{j_z}$, we obviously have
$$z\in V_z:=\{\bar u^*_{({\complement \Omega})\cap B'_{j_z}}>-1\}(\subset \overline{B'_{j_z}})$$
because
$$-1<u^*_{\complement \Omega\cap B'_{j_z}}(z)\leq \bar u^*_{{\complement \Omega}
\cap B'_{j_z}}(z).$$ 
It is clear that  $V_z$ is a plurifinely open set.
On the set $\complement \Omega\cap B'_{j_z}$ we have
$\bar u^*_{{\complement \Omega}\cap B'_{j_z}}=-1$ q.e., and hence
$V_z=\Omega\cap V_z\cup F_z$ for some pluripolar set $F_z\subset B'_{j_z}$.
On the other hand, we have
$$\bigcup_{z\in \Omega} V_z
=\bigcup_{z\in \Omega}\{\bar u^*_{({\complement \Omega})\cap B'_{j_z}}>-1\}
=\bigcup_{j\in J}\{\bar u^*_{({\complement \Omega})\cap B'_j}>-1\},$$
where $J=\{j_z: z\in \Omega\} (\subset \NN)$. For each $j\in J$, there is a point 
$z_j\in \Omega$ such that
$j=j_{z_j}$, so that
$$\Omega\subset \bigcup_{z\in \Omega}V_z=\bigcup_{j\in J}\{\bar u^*_{({\complement \Omega})
\cap B'_{j_{z_j}}}>-1\}
=\bigcup_{j\in J}V_{z_j}.$$
For every $j\in J$ write  $B_j=B'_{j_{z_j}}$ and $\varphi_j=
\bar u^*_{({\complement \Omega})\cap B_j}+1$ and
$P=\bigcup_{j\in J}F_{z_j}$. Then $P$ is pluripolar
and we have $\bigcup\{\varphi_j>0\}\setminus P\subset \Omega\subset \bigcup\{\varphi_j>0\}$
as stated.
\end{proof}

\begin{prop}\label{prop1.5}
Let $D\subset \CC^n$ ($n\geq 1$) be a bounded hyperconvex domain and let
$\Omega$ be a plurifinely open subset of $D$ of the form
$\Omega=\bigcup_j\{\varphi_j>0\}$, where each $\phi_j$ is
a nonnegative plurisubharmonic function
defined on an open ball $B_j\subset D$. Assume that
$u_0,. . . , u_n \in \cal E(D)$ such that $u_0 = u_1$ on $\Omega$.
Then
\begin{equation}\label{eq1.5}
dd^cu_0 \wedge T|_\Omega = dd^cu_1 \wedge T|_\Omega,
\end{equation}
where $T := dd^cu_2 \wedge . . . \wedge dd^cu_n.$
\end{prop}

\begin{proof}
It suffices of course to prove the theorem in case $\Omega=\{\varphi>0\}$,
where $\phi$ is a nonnegative plurisubharmonic function
on an open ball $B\subset D$. Without loss
of generality, we may assume that $u_k \in {\cal F}(D)$, $0 \leq k \leq n$.
Let $\psi \in {\cal E}_0(B)$, $\psi <0$, and define
$u_{j,k}:= \max\{j\psi, u_k|_B\}$, $k = 0, 1, ..., n.$
We easily have
$${\cal E}_0(B) \ni u_{j,k}\downarrow u_k|_B \text{ as } j\uparrow +\infty,
\  k = 0, ...,, n$$
and, by the hypotheses,
\begin{equation}\label{eq1.6}
u_{j,0} = u_{j,1} \text{ on } \Omega,  \forall j \geq 1.
\end{equation}
Write $T_j:= dd^cu_{j,2} \wedge ... \wedge dd^cu_{j,n}.$ Since
the $u_{j,k}$ are bounded plurisubharmonic functions in $B$,
it follows by Corollary 3.4 in \cite{EKW} that
$$dd^cu_{j,0} \wedge T_j = dd^cu_{j,1} \wedge T_j \text{ on } \Omega.$$
and hence
\begin{equation}\label{eq1.7}
\varphi dd^cu_{j,0} \wedge T_j = \varphi dd^cu_{j,1} \wedge T_j \text{ on } B
\end{equation}
because $\varphi=0$ on $B\setminus \Omega$. By Corollary 5.2 in \cite{Ce} we have
$$\varphi dd^cu_{j,0} \wedge T_j \to \varphi dd^c(u_0|_B) \wedge (T|_B)$$
and
$$\varphi dd^cu_{j,1} \wedge T_j \to \varphi dd^c(u_1|_B) \wedge (T|_B)$$
weakly* on $B$ as $j \to +\infty$. Hence, by (\ref{eq1.7}),
$$\varphi dd^c(u_0|_B) \wedge (T|_B) = \varphi dd^c(u_1|_B) \wedge (T|_B) \text{ on } B,$$
and therefore
$$dd^cu_0 \wedge T = dd^cu_1 \wedge T \text{ on } \Omega.$$
\end{proof}

\smallskip
\textit{Proof of Theorem \ref{thm1.1}.}
By Lemma \ref{lemma1.1} there exists a set
$J\subset \NN$ with the property that for each positive $j\in J$
there exists an open ball $B_j\subset \Omega$ and a
function  $\varphi_j\in \PSH_-(B_j)$ such that $\Omega\subset \bigcup_{j\in J}\{\varphi_j>0\}$
and that $P=\bigcup_{j\in J}\{\varphi_j>0\}\setminus \Omega$ is pluripolar.
For each $j\in J$ write $\Omega_j=\{\varphi_j>0\}$. 
Suppose that $u_1=v_1$, ..., $u_n=v_n$ on $\Omega$. The restrictions of
$u_1,...,u_n, v_1,...,v_n$  to $B_j$ are
equal on $\Omega_j\setminus P\subset \Omega$ since they are equal on $\Omega$,
and therefore they are equal on $\Omega_j$ by plurifine  continuity,
$P$ having an empty plurifine interior. We suppose that
$n\geq 2$, the case $n=1$ follows from the fact that
$dd^cu=\Delta u$ (the Laplacian of $u$ in distributional sense)
and \cite[1.XI.18]{D}.
It then follows by Proposition \ref{prop1.5} that
$dd^cu_1\wedge dd^cu_2\wedge...\wedge dd^cu_n |_{\Omega_j}
=dd^cv_1\wedge dd^cu_2\wedge...\wedge dd^cu_n|_{\Omega_j}$.
Let $k$ the larger integer $m\leq n-1$ such that
$$dd^cu_1\wedge dd^cu_2\wedge...\wedge dd^cu_n |_{\Omega_j}
=dd^cv_1\wedge...\wedge dd^cv_m\wedge dd^cu_{m+1}\wedge...\wedge dd^cu_n|_{\Omega_j}.$$
Suppose that $k<n-1$, then
\begin{align*}
dd^cu_1\wedge dd^c&u_2\wedge...\wedge dd^cu_n |_{\Omega_j}
= dd^cv_1\wedge...\wedge dd^cv_k\wedge dd^cu_{k+1}\wedge...\wedge dd^cu_n|_{\Omega_j}\\
&= dd^cu_{k+1}\wedge dd^cv_1\wedge...\wedge dd^cv_k\wedge dd^cu_{k+2}\wedge...\wedge dd^cu_n|_{\Omega_j}\\
&= dd^cv_{k+1}\wedge dd^cv_1\wedge...\wedge dd^cv_k\wedge dd^cu_{k+2}\wedge...\wedge dd^cu_n|_{\Omega_j}\\
&= dd^cv_1\wedge...\wedge dd^cv_m\wedge dd^cv_{k+1}\wedge...\wedge dd^cu_n|_{\Omega_j}
\end{align*}
The third equality holds by Proposition \ref{prop1.5}. 
But this contradicts the definition of $k$. It then follows 
that $k=n-1$ so that 
$$dd^cu_1\wedge dd^cu_2\wedge...\wedge dd^cu_n |_{\Omega_j}
=dd^cv_1\wedge...\wedge dd^cv_{n-1}\wedge dd^cu_n|_{\Omega_j}$$ 
and hence 
$$dd^cu_1\wedge dd^cu_2\wedge...\wedge dd^cu_n |_{\Omega_j}
=dd^cv_1\wedge...\wedge dd^cv_{n-1}\wedge dd^cv_n|_{\Omega_j}$$ 
again according to Proposition \ref{prop1.5} applied to the right hand of the above equality.
The proof is now complete.

\end{document}